\documentclass[11pt,a4paper]{article}
\usepackage{a4wide}
\usepackage{latexsym}
\usepackage{amsmath}
\usepackage{amssymb}
\usepackage{amsthm}

\usepackage{color}

\theoremstyle{plain}
\newtheorem{theo}{Theorem}[section]
\newtheorem{lemma}[theo]{Lemma}
\newtheorem{prop}[theo]{Proposition}
\newtheorem{coro}[theo]{Corollary}

\theoremstyle{definition}
\newtheorem{defi}[theo]{Definition}
\newtheorem{rema}[theo]{Remark}
\newtheorem{exam}[theo]{Example}

\def\a{\alpha}

\def\ga{\gamma}
\def\o{\omega}
\def\C{\mathbb{C}}
\def\N{\mathbb{N}}

\def\R{\mathbb{R}}
\def\Z{\mathbb{Z}}

\def\M{\textbf{M}}
\def\bM{\textbf{M}}

\begin{document}

\title{On the summability of a class of formal power series}
\author{A. Lastra, J. Sanz and J. R. Sendra}

\maketitle

\begin{center}
{\bf Abstract}
\end{center}

The formal power series solutions for some classes of moment differential equations, induced by polynomial moment differential operators, are characterized in terms of their summability properties, and in terms of estimates for recursive expressions involving their coefficients. Of special interest are the particularization of these results to classes of fractional  and of ordinary differential equations.
The Stokes' phenomenon can be described in some of these situations.
The main results are extended into the framework of $q$-Gevrey asymptotics and $q$-difference equations.\vspace{0.5cm}


\section{Introduction}

In 1886, H. Poincar\'e put forward the concept of asymptotic expansion at 0 for holomorphic functions defined in an open sector in $\C$ with vertex at the origin. He intended to give an analytic meaning to formal (in general, divergent) power series solutions of ordinary differential equations at irregular singular points.

In the late 1970's, J. P. Ramis~\cite{ramis1,ramis2} developed the theory of $k$-summability of formal power series, proving that every formal power series solution to a linear system of meromorphic ordinary differential equation in the complex domain at an irregular singular point can be decomposed into a finite product of formal power series, each of which turns out to be $k$-summable, for some order $k>0$ which depends on the series.

The previous result was only obtained in a theoretical way, and it was improved by J. Ecalle~\cite{ecalle1,ecalle2} through the introduction of multisummability. It turned out that any formal power series solution of a (linear or not) meromorphic system of ordinary differential equations at an irregular singularity is, indeed, multisummable~\cite{braaksma,ramissibuya,balser}, and an algorithm is available in order to compute actual (i.e., analytic) solutions departing from formal ones. Moreover, B. Malgrange gave a method based on the Newton polygon in order to find the exact orders involved in the multisummability process.


It is worth mentioning a recent paper of O. Costin and X. Xia~\cite{CostinXia}, where a criteria is given on the Taylor coefficients for the associated analytic function not to have natural boundaries and to belong to the class of functions analytic in the complex plane with finitely many cuts and with algebraic behavior at infinity.
Their condition is that the coefficients admit generalized Ecalle-Borel summable transseries, a property shared by solutions of many general classes of problems in analysis. In particular, this is the case whenever the coefficients solve a generic linear or nonlinear recurrence relation of finite order with analytic coefficients, see~\cite{Costin,braaksma0,ecalle3}. Such recurrence relations exist for instance when the coefficients are obtained by solving differential equations by power series.
This paper intends to be a contribution to such kind of problems.

In the present study, we provide a simple property of a formal power series in order to be summable along certain well chosen directions: its coefficients, when inserted in a given recurrence relation of finite order, provide values whose growth may be suitably controlled. Moreover, we describe a family of differential, moment-differential (including some fractional differential equations involving Caputo's fractional derivatives) and $q$-difference equations for which such a formal power series may appear as a solution, and the corresponding summability procedures, namely $k$-summability, summability with respect to a sequence admitting a nonzero proximate order, and $q$-Gevrey summability, respectively.

We will also comment on the possibility of describing the Stokes' phenomenon in some of these situations, whenever the structure of the singularities of the Borel transform of the formal solutions is simple. So, the difference between neighboring solutions at both sides of a singular direction may be explicitly computed.


The layout of this work is as follows.
Section~\ref{sectPrelim} is mainly devoted to give a brief summary of the main definitions and results concerning general asymptotic expansions of functions defined in sectors of the Riemann surface of the logarithm, and the related concept and technique of summability in a direction, all with respect to a sequence of positive real numbers admitting a nonzero proximate order (see~\cite{lastramaleksanz15}). This generalizes the classical Gevrey asymptotic theory and the corresponding $k$-summability of formal power series, developed by J.-P. Ramis~\cite{ramis1,ramis2}.
In Section~\ref{sectCharacSummabODE}, we describe a family of formal power series appearing as formal solutions of certain moment differential equations, introduced by W. Balser and M. Yoshino~\cite{balseryoshino}, see Theorem~\ref{teoEqM}. As a natural application of these kind of problems, we may mention the case when the sequence of moments gives rise to the so-called fractional Caputo's derivatives, and so one enters the framework of fractional differential equations. 
Next, in Theorem~\ref{teo305ordens} we extend our considerations to a particular type of formal power series of some Gevrey order $s\in\N$, thanks to the crucial fact, due to W. Balser, that the termwise product of sequences of moments of some Gevrey order is again a sequence of moments.

In Section~\ref{sectqGevrey} we indicate the main facts that allow us to provide new insights regarding some formal power series solutions of a class of $q$-difference equations. Here, the $q$-Gevrey asymptotics and $q$-summability theory, mainly studied by J. P. Ramis and C. Zhang (see~\cite{zhang}), play a prominent role. Unlike in the previous approaches, there is no unique sum naturally associated with a formal $q$-Gevrey series in a direction, what leads to the specification of a natural sum through a prescribed variation.

\section{Preliminaries}\label{sectPrelim}

\subsection{Notation}

Let $\mathcal{R}$ denote the Riemann surface of the logarithm.
Let $\theta>0$, $d\in\R$ and $r>0$. We write $S_d(\theta,r)$ for the bounded sector with vertex at the origin, opening $\theta\,\pi$ and bisecting direction $d$, given by
$$S_{d}(\theta,r)=\{z\in\mathcal{R}:|\hbox{arg}(z)-d|<\frac{\theta\pi}{2},0<|z|<r\}.$$
We also consider unbounded sectors
$$S_{d}(\theta):=\{z\in\mathcal{R}:|\hbox{arg}(z)-d|<\frac{\theta\pi}{2}\}.$$
A sectorial region $G_d(\a)$ with bisecting direction $d\in\R$ and opening $\alpha\pi$ will be a domain in $\mathcal{R}$ such that $G_d(\a)\subset S_d(\a)$, and
for every $\beta\in(0,\a)$ there exists $\rho=\rho(\beta)>0$ with $S_d(\beta,\rho)\subset G_d(\a)$. In particular, sectors are sectorial regions.\par\noindent
A sector $T$ is a bounded proper subsector of a sectorial region $G$ (denoted by $T\prec G$) whenever the radius of $T$ is finite and $\overline{T}\subset G$ (the closure is considered in $\mathcal{R}$).

We write $\N_0=\{0,1,2,\dots\}$ and $\N=\{1,2,\dots\}$. $\mathcal{O}(S)$ stands for the set of holomorphic functions in $S$, and $\C[[z]]$ is the set of formal power series with complex coefficients.

\subsection{Summability of formal power series}

The theory of summability of formal solutions of different kinds of functional equations (differential, difference, $q$-difference, etc.) is intimately related to
asymptotics.
This section is devoted to a general overview of both asymptotics and summability of formal power series in a direction, with respect to a sequence of positive numbers admitting a nonzero proximate order. We provide the results without proof, which can be found in
~\cite{javinuevo,lastramaleksanz15,JimenezSanzSchindlLCSNPO,JimenezSanzSchindlInjSurj}. The classical notions of Gevrey asymptotics and $k$-summability are a particular case in this framework, see~\cite{ramis1,ramis2,balser} and Remark~\ref{remaKsummability}.

In what follows, $\bM=(M_{p})_{p\in\N_0}$ stands for a sequence of positive real numbers.
%
%
%
%


\begin{defi}
Let $G$ be a sectorial region with vertex at the origin and $f\in\mathcal{O}(G)$. We say $f$ admits $\widehat{f}(z)=\sum_{n=0}^{\infty}a_nz^n\in\C[[z]]$ as its $\bM$-asymptotic expansion in $G$ if for every $T\prec G$ there exist $A=A(T)>0$ and $C=C(T)>0$ such that for every $n\in\N_0$ one has
$$\big|f(z)-\sum_{p=0}^{n-1}a_pz^p\big|\le CA^nM_n|z|^n,\quad z\in T.$$
\end{defi}

We write $f\sim_{\bM}\sum_{n=0}^{\infty}a_nz^n$ in $G$. $\widetilde{\mathcal{A}}_{\M}(G)$ stands for the linear space of
functions admitting $\bM$-asymptotic expansion in $G$.
%
%
Accordingly, we define the linear space of formal power series
$$\C[[z]]_{\M}=\Big\{\widehat{f}(z)=\sum_{n=0}^{\infty}a_nz^n: \textrm{there exist $C,A>0$ with }|a_n|\le CA^{n}M_{n},\ n\in\N_{0}\Big\}.$$
The linear map $\widetilde{\mathcal{B}}:\widetilde{\mathcal{A}}_{\M}(G)\longrightarrow \C[[z]]_{\M}$ sending a function to its $\M$-asymptotic expansion 
will be called the \textit{asymptotic Borel map}. It is a
homomorphism of algebras if $\M$ is \textit{logarithmically convex} (for short, (lc)), that is, $M_{p}^{2}\le M_{p-1}M_{p+1}$ for every $p\in\N$.

As a consequence of Taylor's formula and Cauchy's integral formula for the derivatives, we have the following result (see \cite{balser} for a proof in the Gevrey case, which may be easily adapted to this more general situation).

\begin{prop}\label{propcotaderidesaasin}
Let $G$ be a sectorial region and $f\in\mathcal{O}(G)$. Then, $f\in\widetilde{\mathcal{A}}_{\M}(G)$ if, and only if,
for every $T\prec G$ there exist $C_T,A_T>0$ such that for every $p\in\N_0$ and $z\in T$, one has $|f^{(p)}(z)|\le C_TA_T^p p!M_p$.
\end{prop}


\begin{defi}
A function $f\in\widetilde{\mathcal{A}}_{\M}(G)$ is said to be \textit{flat} (or $\M$-\textit{flat}) if $\widetilde{\mathcal{B}}(f)$ is the null series, i.e., $f\sim_{\M}\widehat{0}$.
We say that $\mathcal{\widetilde{A}}_{\M}(G)$ is  \textit{quasianalytic} if it does not contain nontrivial flat functions; in other words,
the asymptotic Borel map is injective in the class.
\end{defi}

The study of quasianalyticity, for ultraholomorphic classes subject to uniform bounds either for the derivatives (as suggested by Proposition~\ref{propcotaderidesaasin}) or for the asymptotics, has been made by several authors, both in one~\cite{mandel,Salinas} or several variables~\cite{groening,lastrasanzquasi}.
In order to avoid trivial situations, we restrict ourselves to \textit{weight sequences}, i.e., (lc) sequences with $\lim_{p\to\infty}M_{p+1}/M_p=\infty$.


The concept of proximate order, relevant in the theory of growth of holomorphic functions in sectors (see, for example,~\cite{GoldbergOstrowskii}), plays a prominent role in the following result, which appeared in a weaker form in~\cite{javinuevo}.

\begin{coro}[Watson's Lemma,\ \cite{JimenezSanzSchindlInjSurj}]\label{coroWatsonlemma}
Given a weight sequence $\M$ and $\ga>0$, $\widetilde{\mathcal{A}}_{\M}(G_0({\gamma}))$ is quasianalytic if, and only if,  $\gamma>\omega(\M)$, where
\begin{equation*}
\omega(\bM)=\liminf_{p\to\infty}
\frac{\log(M_{p+1}/M_p)}{\log(p)}\in[0,\infty].
\end{equation*}

\end{coro}

We are ready for the definition of summability in a direction in this context.

\begin{defi}[\cite{lastramaleksanz15}]\label{defisumable}
Let $d\in\R$ and $\M$ be a weight sequence.
We say $\widehat{f}=\sum_{p\ge 0}\displaystyle a_pz^p$ is \textit{$\M$-summable in direction $d$} if there exist a sectorial region $G=G_d(\ga)$, with $\ga>\o(\M)$, and
a function $f\in\widetilde{\mathcal{A}}_{\M}(G)$ such that $f\sim_{\M}\widehat{f}$.
\end{defi}

According to Watson's Lemma, $f$ is unique with the property stated and will be called
the $\M$-sum of $\widehat f$ in direction $d$. In order to explicitly construct this sum, we need to introduce an auxiliary function
$\omega_{\bM}:(0,\infty)\to\R$ given by
\begin{equation*}
\omega_{\bM}(t)=
\sup_{p\in\N_0}\log\big(\frac{t^p}{M_p}\big).
\end{equation*}

As it may be found in~\cite{lastramaleksanz15,JimenezSanzSchindlLCSNPO}, whenever $\M$ admits 
a nonzero proximate order $\rho(t)$, i.e., there exist constants $A,B>0$ with
$$
A\le \frac{\o_{\M}(t)}{t^{\rho(t)}}
\le B,\quad t\textrm { large enough},
$$
one may construct pairs of $\M$-summability kernel functions $e(z)$ and $E(z)$, and the moment function associated with $e(z)$, that is,
$$m(\lambda):=\int_{0}^{\infty}t^{\lambda-1}e(t)dt,\quad \Re(\lambda)>0.$$
The so-called sequence of moments $\mathfrak{m}=(m(p))_{p\in\N_0}$ turns out to be equivalent to $\M$ (in the sense that there  exist $B,C>0$ such that $B^pM_p\le m(p)\le C^pM_p$ for every $p\in\N_0$). Then, suitable Laplace- and Borel-like formal and analytic transforms allow for the reconstruction of the sum, as the next result shows.

\begin{theo}
Suppose the sequence $\M$ admits a nonzero proximate order, $d$ is a direction and $\widehat f=\sum_{n\ge 0}\displaystyle a_nz^n$. The following are equivalent:
\begin{itemize}
\item[(i)] $\widehat f$ is $\M$-summable in direction $d$.
\item[(ii)] For every (some) kernel $e$ of $\M$-summability, its formal Borel transform $g:=\widehat{\mathcal{B}}_{\mathfrak{m}}(\widehat{f}):=\sum_{n\ge 0}\displaystyle\frac{a_n}{m(n)}z^n$ converges, i.e. has positive radius of convergence,
    it admits analytic continuation in an unbounded (narrow enough) sector $S$ bisected by $d$, and moreover is of $\M$-growth in $S$, i.e., for every unbounded subsector $T$ of $S$ there exist $k_1,k_2>0$ such that for every $z\in T$,
\begin{equation}\label{crecimientoM}
|g(z)|\le 
k_1\exp(\o_{\M}(k_2|z|)).
\end{equation}
\end{itemize}
In case any of the previous holds, the $\M$-sum of $\widehat f$ in direction $d$ can be constructed as an $\bM$-analogue of Laplace transform,
$$f(z)=\int_{0}^{\infty(d)}e(u/z)g(u)\frac{du}{u},\quad |\arg(z)-d|<\omega(\bM)\pi/2,\ |z|\hbox{ small enough.}$$
\end{theo}

The previous integral is the line integral along the path parameterized by $s\in(0,\infty)\mapsto se^{id}$. The analytic continuation of $f$ to a sectorial region bisected by $d$ and with opening larger than $\omega(\bM)\pi$ is obtained by changing the line of integration into neighboring directions $d'$ within $S$.
%
%


\begin{exam}\label{remaKsummability}
For $k>0$ and the Gevrey sequence $\M_{1/k}=(p!^{1/k})_{p\in\N_0}$, the classical Gevrey asymptotic theory and the $k$-summability method, introduced by J.-P. Ramis~\cite{ramis1,ramis2,balser}, are obtained. In this case, $\omega(\M_{1/k})=1/k$, the function $\omega_{\M_{1/k}}(t)$ grows like $t^k$ at infinity, and $\M_{1/k}$ admits the constant proximate order $\rho(t)=k$. A kernel of $k$-summability is $e(z)=kz^{k}\exp(-z^k)$, for which the sequence of moments is
$\mathfrak{m}=(\Gamma(1+p/k))_{p\in\N_0}$, where $\Gamma$ denotes the Eulerian Gamma function. So, a series $\widehat{f}(z)=\sum_{p\ge0}f_{p}z^{p}$ is $k$-summable in direction $d$ if, and only if, its formal Borel transform of order $k$, $\widehat{\mathcal{B}}_k\widehat{f}(z):=\sum_{p\ge0}\frac{f_{p}}{\Gamma(1+p/k)}z^{p}$,
has positive radius of convergence, its sum $g$ admits analytic continuation into an unbounded sector $S$ bisected by $d$, and there exist $c_1,c_2>0$ with
\begin{equation}\label{e163}
|g(z)|\le c_1\exp\left(c_2|z|^{k}\right),\quad z\in S.
\end{equation}
If these statements hold, then the $k$-sum of $\widehat f$ in direction $d$ is given by
\begin{equation}\label{equakSumDirectiond}
f(z)=kz^{-k}\int_{0}^{\infty(d)}g(u)e^{-(u/z)^{k}}u^{k-1}du
\end{equation}
whenever both sides of the equality are defined.
\end{exam}

\begin{rema}
Weight sequences admitting a proximate order are strongly regular (in the terminology of V. Thilliez): they are (lc), of moderate growth (there exists $A>0$ such that
$M_{p+q}\le A^{p+q}M_{p}M_{q}$ for all $p,q\in\N_0$), and satisfy the strong non-quasianalyticity condition, i.e., there exists $B>0$ such that
$$
\sum_{q\ge p}\frac{M_{q}}{(q+1)M_{q+1}}\le B\frac{M_{p}}{M_{p+1}},\qquad p\in\N_0.
$$
These conditions are quite classical and they naturally appear in the literature~\cite{komatsu} concerning the study of so-called Carleman ultradifferentiable or ultraholomorphic classes, consisting of $\mathcal{C}^{\infty}$ (respectively, analytic) functions whose derivatives' growth is controlled mainly by the given sequence $\bM$.

Every example of strongly regular sequence appearing in the applications admits a nonzero proximate order. This is the case for the sequence $(p!^{\alpha}\prod_{m=0}^{p}\log^\beta(e+p))_{p\ge0}$, where $\alpha>0$ and $\beta\in\R$, which is strongly regular (in case $\beta<0$, after some inessential changes like scaling or adjusting a finite number of terms).
However, there do exist strongly regular sequences not admitting such a proximate order, see~\cite{JimenezSanzSchindlLCSNPO}.

For $q>1$ the sequence $(q^{\frac{p(p-1)}{2}})_{p\ge 0}$, which will be considered later, is not strongly regular as it does not have moderate growth, and so it does not admit a nonzero proximate order.
\end{rema}

\section{Characterization of the summability of certain formal solutions of ODEs}\label{sectCharacSummabODE}

The forthcoming generalization of the derivative operator was firstly studied by W. Balser and M. Yoshino~\cite{balseryoshino}: Given a sequence of moments $\mathfrak{m}=(m(p))_{p\ge0}$ as before, one can define the moment differential operator $\partial_{\mathfrak{m}}$ on formal power series with complex coefficients by
$$\partial_{\mathfrak{m}}\left(\sum_{p\ge0}\frac{f_p}{m(p)}z^p\right)= \sum_{p\ge0}\frac{f_{p+1}}{m(p)}z^p,$$
and this definition may be proved (see~\cite{Michalik1,lastramaleksanz15}) to extend to functions holomorphic around the origin. Observe that this operator coincides with the usual derivative when working with the sequence $m(p):=p!=\Gamma(1+p)$, moments of the kernel function $e(z)=ze^{-z}$, corresponding to 1-summability.

Moment partial differential equations, related to these operators, have already been studied in~\cite{balseryoshino} and by S. Michalik~\cite{michalik, Michalik5}, in the case that $(m(p))_{p\in\N_0}$ is equivalent to a Gevrey sequence, and in the general case in \cite{lastramaleksanz15}.

As a matter of fact, the particularization of the sequence of moments in the form $$\mathfrak{m}^{1/k}:=\Big(\Gamma\left(1+\frac{p}{k}\right)\Big)_{p\in\N_0},
$$
for some $k\in\N$, is intimately related to the Caputo fractional derivative of order $1/k$, denoted by $\partial_{x}^{1/k}$. Indeed, $(\partial_{\mathfrak{m}^{1/k}}\widehat{f})(z^{1/k})$ coincides with $\partial_{x}^{1/k}(\widehat{f}(x^{1/k}))$ for every formal power series $f$. In particular, one has $\partial_{z}^{1/k}z^{0/k}=0$, and
$$\partial_{z}^{1/k}z^{n/k}=\frac{\Gamma(1+\frac{n}{k})}{\Gamma(1+\frac{n-1}{k})}z^{(n-1)/k},$$
for every $n\in\N$. The action of this operator coincides with that of Caputo's derivative. Following~\cite{kilbas}, Caputo derivative, which is defined for $\alpha\notin\N_0$, $n\in\N$, $n-1<\alpha<n$, by
$$^{C}D^{\alpha}_{0+}\varphi(z):= \frac{1}{\Gamma(n-\alpha)} \int_{0}^{z}\frac{\varphi^{(n)}(t)}{(z-t)^{\alpha-n-1}}dt,$$
turns out to satisfy
$$^{C}D^{\alpha}_{0+}z^{\beta}= \frac{\Gamma(1+\beta)}{\Gamma(1+\beta-\alpha)}z^{\beta-\alpha},\quad n-1<\beta,\ \beta\in\N,
$$
and
$$
^{C}D^{\alpha}_{0+}z^{k}=0,\quad k=0,\ldots,n-1.$$


The next results provide new insight for some kind of ordinary moment differential equations.

\begin{lemma}
Let $\bM$ admit a nonzero proximate order, $e$ be a kernel function associated with $\bM$ and  $\mathfrak{m}=(m(p))_{p\in\N_0}$ be the sequence of its moments.
Then, for every formal power series $\widehat{f}\in\C[[z]]$ one has
$$\tau\widehat{\mathcal{B}}_{\mathfrak{m}}(\widehat{f})(\tau)= \widehat{\mathcal{B}}_{\mathfrak{m}}
\big(z\partial_\mathfrak{m}(z\widehat{f}(z)\big)(\tau).$$
\end{lemma}


The proof follows straightforward from the definitions of $\widehat{\mathcal{B}}_{\mathfrak{m}}$ and $\partial_\mathfrak{m}$. In the forthcoming statement, for a monomial $\lambda^n$ the expression $\big(z\partial_\mathfrak{m}(z\,\cdot\,)\big)^n$ is to be understood as the operator computing the $n$-th iteration of the operator $z\partial_\mathfrak{m}(z\,\cdot\,)$ that multiplies by $z$, then applies $\partial_\mathfrak{m}$ and again multiplies by $z$.

\begin{theo}\label{teoEqM}
Let $\bM$,  $e$ and $\mathfrak{m}=(m(p))_{p\ge0}$ be as before.
Let $\widehat{f}(z)=\sum_{p\ge1}f_{p}z^{p}\in\C[[z]]$ be a formal power series. The following statements are equivalent:
\begin{enumerate}
\item[(1)] $\widehat{f}$ is the formal solution of $P\big(z\partial_\mathfrak{m}(z\,\cdot\,)\big)y=g(z)$, for some polynomial $P\in\C[z]$ with $P(0)\neq 0$ and some $g\in\C\{z\}$.
\item[(2)] There exist $r\in\N$, complex numbers $a_1,\dots,a_r$ and positive constants $C,M$ such that
\begin{equation}\label{e228bis}
\Big|\frac{f_{j}}{m(j)}-\sum_{k=j-r}^{j-1}a_{j-k}\frac{f_k}{m(k)}\Big|\le \frac{CM^j}{m(j)}
\end{equation}
holds for every $j>r$.
\item[(3)] $\widehat{f}$ is $\M$-summable in any direction $d$ but the arguments of the roots of a polynomial
$$h(z)=1-a_1z-\dots-a_rz^r,$$
and its sum, say $f\in\mathcal{O}(S_d)$ for some sector $S_d$ of opening larger than $\pi\o(\M)$, is an actual solution of the problem $P\big(z\partial_\mathfrak{m}(z\,\cdot\,)\big)y=g(z)$, for some polynomial $P\in\C[z]$, with $P(0)\neq 0$, and some $g\in\C\{z\}$ which do not depend on the choice of $d$.
\end{enumerate}
\end{theo}

\begin{proof}
$(1)\Rightarrow(2)$ We assume that $\widehat{f}$ is a formal solution of $P\big(z\partial_\mathfrak{m}(z\,\cdot\,)\big)y=g(z)$ for some polynomial $P\in\C[z]$, with $P(0)\neq 0$, and some $g\in\C\{z\}$. Up to multiplication by a constant factor, we may assume without loss of generality that $P(z)=1-\sum_{p=1}^{r}a_pz^p$ and $g(z)=\sum_{p=0}^{\infty}b_{p}z^p$.
It is clear that
$$(z\partial_\mathfrak{m}(z\,\cdot\,))\left(\sum_{p\ge1}f_pz^p\right)= \sum_{p\ge2}\frac{m(p)f_{p-1}}{m(p-1)}z^p$$
and, by induction on $k\in\N$,
$$(z\partial_\mathfrak{m}(z\,\cdot\,))^k\left(\sum_{p\ge1}f_pz^p\right)=
\sum_{p\ge k+1}\frac{m(p)f_{p-k}}{m(p-k)}z^p,$$
for every $k\ge2$.
If one plugs $\widehat{f}$ into the equation, then for $j> r$ the coefficient of $z^j$ in $P\big(z\partial_\mathfrak{m}(z\,\cdot\,)\big)\widehat{f}$ is
$$f_j-m(j)\sum_{k=j-r}^{j-1}a_{j-k}\frac{f_k}{m(k)},$$
which should equal $b_{j}$.
Since $g$ converges, there exist $C,K>0$ with
$|b_p|\le CK^p$ for every $p\in\N$, and so for every $j>r$,
\begin{equation*}
\Big|\frac{f_{j}}{m(j)}-\sum_{k=j-r}^{j-1}a_{j-k}\frac{f_k}{m(k)}\Big|\le \frac{CK^j}{m(j)},
\end{equation*}
as desired.\par\noindent
$(2)\Rightarrow(3)$ Consider the formal Borel transform of $\widehat{f}$,
$\widehat{\mathcal{B}}_{\mathfrak{m}}\widehat{f}(z)=\sum_{p\ge 1}f_{p+1}m(p+1)z^p/m(p)$, and the polynomial
$h(z)=1-a_1z-\dots-a_rz^r$. If we put $f_0:=0$,\dots,$f_{1-r}:=0$, it is clear that
$$h(z)\cdot\widehat{\mathcal{B}}_{\mathfrak{m}}\widehat{f}(z)=\sum_{j\ge 1}\Big(\frac{f_{j}}{m(j)}-\sum_{k=j-r}^{j-1}a_{j-k}\frac{f_k}{m(k)}\Big)z^j.$$
The estimates in (\ref{e228bis}) imply that the series on the right-hand side, say $b(z)=\sum_{j\ge 1}b_jz^j/m(j)$, defines an entire function which, taking into account the equivalence of $\M$ and the sequence of moments, and by the definition of the auxiliary function $\omega_{\M}(t)$, turns out to satisfy estimates as those in \eqref{crecimientoM}. The same is true then for the function $b(z)/h(z)$ (holomorphic in $\C$ with the roots of $h$ deleted, in particular, in a neighborhood of the origin) in any open unbounded sector with vertex at 0 and not containing any of these roots. So, $\widehat{f}$ is $\M$-summable in every direction $d$ which is not an argument of any of the roots of $h$, as desired. Let $f$ be the $\M$-sum of $\widehat{f}$ in an appropriate direction $d\in\R$, defined in a sector $S_d$ of opening larger than $\pi\o(\M)$. According to previous computations, the coefficient of $z^j$ in $h(z\partial_\mathfrak{m}(z\,\cdot\,))\widehat{f}$ is
$$f_{j}-m(j)\sum_{k=j-r}^{j-1}a_{j-k}\frac{f_k}{m(k)}=b_j.$$
Put $g(z):=\sum_{j\ge 1}b_jz^j$, which belongs to $\C\{z\}$.
We have that $h(z\partial_\mathfrak{m}(z\,\cdot\,))\widehat{f}(z)-g(z)$ is the null formal power series. The stability properties of $\M$-asymptotic expansions with respect to sums and products guarantee that $h(z\partial_\mathfrak{m}(z\,\cdot\,))f(z)-g(z)$ admits null $\M$-asymptotic expansion in $\widetilde{S}_d$, and by Watson's Lemma we conclude that $f$ solves the equation $h(z\partial_\mathfrak{m}(z\,\cdot\,))y=g(z)$ in $S_d$.\par\noindent
$(3)\Rightarrow(1)$ Since $\widehat{f}$ is the asymptotic expansion of $f$ in a suitable sector, it is readily deduced from the elementary properties of asymptotic expansions that, whenever $f$ solves the differential equation, $\widehat{f}$ also solves the corresponding formal one.
\end{proof}

We may easily deduce the following corollary, treating the case of a polynomial inhomogeneous term.

\begin{coro}\label{coroEqMpolynomial}
Let $\bM$,  $e$ and $\mathfrak{m}=(m(p))_{p\ge0}$ be as before.
Let $\widehat{f}(z)=\sum_{p\ge1}f_{p}z^{p}\in\C[[z]]$ be a formal power series. The following statements are equivalent:
\begin{enumerate}
\item[(1)] $\widehat{f}$ is the formal solution of $P\big(z\partial_\mathfrak{m}(z\,\cdot\,)\big)y=Q(z)$, for some polynomial $P,Q\in\C[z]$ with $P(0)\neq 0$.
\item[(2)] There exist $r\in\N$ and complex numbers $a_1,\dots,a_r$ such that
\begin{equation}\label{equa_coeffic_polynomial}
\frac{f_{j}}{m(j)}=\sum_{k=j-r}^{j-1}a_{j-k}\frac{f_k}{m(k)}
\end{equation}
holds for every $j>r$.
\item[(3)] $\widehat{f}$ is $\M$-summable in any direction $d$ but the arguments of the roots of a polynomial
$$h(z)=1-a_1z-\dots-a_rz^r,$$
and its sum, say $f\in\mathcal{O}(S_d)$ for some sector $S_d$ of opening larger than $\pi\o(\M)$, is an actual solution of the problem $P\big(z\partial_\mathfrak{m}(z\,\cdot\,)\big)y=Q(z)$, for some polynomials $P,Q\in\C[z]$ with $P(0)\neq 0$, which do not depend on the choice of $d$.
\end{enumerate}
\end{coro}

\begin{proof}
$(1)\Rightarrow(2)$ Without loss of generality, we may write $P(z)=1-\sum_{p=1}^{r_0}a_pz^p$, with $a_{r_0}\neq 0$, and $Q(z)=\sum_{p=0}^{s}b_{p}z^p$, with $b_s\neq 0$. As in the previous proof, for $j> r_0$ the coefficient of $z^j$ in $P\big(z\partial_\mathfrak{m}(z\,\cdot\,)\big)\widehat{f}$ is
$$f_j-m(j)\sum_{k=j-r}^{j-1}a_{j-k}\frac{f_k}{m(k)},$$
which should equal $0$ as long as $j>s$. So, in case $s\le r_0$ we put $r:=r_0$, and (\ref{equa_coeffic_polynomial}) holds for every $j>r$. Whenever $s>r_0$ we define $r:=s$ and $a_{r_0+1}:=0,\dots,a_s:=0$, in such a way that (\ref{equa_coeffic_polynomial}) again holds for every $j>r$.\par\noindent
$(2)\Rightarrow(3)$ Consider the polynomial
$h(z)=1-a_1z-\dots-a_rz^r$. If we put $f_0:=0$,\dots,$f_{1-r}:=0$, it is clear by~\eqref{equa_coeffic_polynomial} that
$$h(z)\cdot\widehat{\mathcal{B}}_{\mathfrak{m}}\widehat{f}(z)=\sum_{j= 1}^r\Big(\frac{f_{j}}{m(j)}-\sum_{k=j-r}^{j-1}a_{j-k}\frac{f_k}{m(k)}\Big)z^j,$$
so that the function $\widehat{\mathcal{B}}_{\mathfrak{m}}\widehat{f}$ is a rational function, holomorphic in $\C$ except for the roots of $h$, and with exponential growth order 0 in any open unbounded sector with vertex at 0 and not containing any of these roots. So, $\widehat{f}$ is $\M$-summable in every direction $d$ which is not an argument of any of the roots of $h$. Consider the polynomial
$$Q(z)=\sum_{j= 1}^r\Big(f_{j}-m(j)\sum_{k=j-r}^{j-1}a_{j-k}\frac{f_k}{m(k)}\Big)z^j.$$
As before, Watson's lemma guarantees that its sum solves the equation $P(z\partial_\mathfrak{m}(z\,\cdot\,))y=Q(z)$ in a suitably wide sector $S_d$, bisected by $d$, where it is defined.\par\noindent
$(3)\Rightarrow(1)$ Immediate.
\end{proof}

The rest of this section is devoted to the study of necessary and sufficient conditions for a formal power series to be a solution of a family of ordinary differential equations induced by some polynomial differential operators. As we will see, it is natural to restrict attention to values $s:=1/k\in\N$ for the Gevrey order. Moreover, this formal solution turns out to be $k$-summable along appropriate directions which are determined.

For this purpose, we give a property of formal Borel transforms which can be found in the literature (see Proposition 4 in~\cite{threefold}, for instance). Its proof is direct, so we omit it.

\begin{lemma}\label{lemaBorelformalordens}
For every $\widehat{f}(z)\in\C[[z]]$, consider the modified Borel transform of order $k=1/s$, where $s\in\N$, given by
$$
\widetilde{\mathcal{B}}_{k}\widehat{f}(z)=\sum_{p\ge0}\frac{f_{p}}{p!^{s}}z^{p}
$$
(observe that $\widetilde{\mathcal{B}}_{1}=\widehat{\mathcal{B}}_{1}$). The following formal equality holds:
$$\tau\widetilde{\mathcal{B}}_{k}(\widehat{f})(\tau)= \widetilde{\mathcal{B}}_{k}\big(z(z\partial_z+1)^s\widehat{f}(z)\big)(\tau).$$
\end{lemma}

\begin{rema}\label{remaordens}
For what follows, we need to take into account that not only the sequences $(\Gamma(1+sp))_{p\in\N_0}$ and $(p!^s)_{p\in\N_0}$ are equivalent, but the second one is also a sequence of moments for a suitable kernel of $k$-summability, as defined by W. Balser~\cite[Section\ 5.5]{balser}. This last statement stems from~\cite[Theorem\ 31]{balser}, since $(p!)_{p\in\N_0}$ is a sequence of moments,
and any other sequence which may be expressed as a finite term-by-term product of several moment sequences is again a moment sequence. Moreover, a series $\widehat{f}(z)=\sum_{p\ge0}f_{p}z^{p}$ is $k$-summable in direction $d$ if, and only if, its modified Borel transform of order $k$, $\widetilde{\mathcal{B}}_k\widehat{f}(z)$, satisfies the same properties indicated in Example~\ref{remaKsummability} for the classical Borel transform, in particular the estimates in~\eqref{e163}.
\end{rema}

We are ready to state our next two results, which are particular versions of Theorem~\ref{teoEqM} and Corollary~\ref{coroEqMpolynomial}, respectively.

\begin{theo}\label{teo305ordens}
Let $\widehat{f}(z)=\sum_{p\ge1}f_{p}z^{p}\in\C[[z]]$ be a formal power series and $s=1/k\in\N$. The following statements are equivalent:
\begin{enumerate}
\item[(1)] $\widehat{f}$ is the formal solution of $P\big(z(z\partial_z+1)^s\big)y=g(z)$, for some polynomial $P\in\C[z]$ with $P(0)\neq 0$ and some $g\in\C\{z\}$.
\item[(2)] There exist $r\in\N$, complex numbers $a_1,\dots,a_r$ and positive constants $C,M$ such that
\begin{equation}\label{e228ordens}
\Big|\frac{f_{j}}{j!^s}-\sum_{k=j-r}^{j-1}a_{j-k}\frac{f_k}{k!^s}\Big|\le \frac{CM^j}{j!^s}
\end{equation}
holds for every $j>r$.
\item[(3)] $\widehat{f}$ is $k$-summable in any direction $d$ but the arguments of the roots of a polynomial
$$h(z)=1-a_1z-\dots-a_rz^r,$$
and its sum, say $f\in\mathcal{O}(S_d)$ for some sector $S_d$ of opening larger than $\pi/k$, is an actual solution of the problem $P\big(z(z\partial_z+1)^s\big)y=g(z)$, for some polynomial $P\in\C[z]$, with $P(0)\neq 0$, and some $g\in\C\{z\}$ which do not depend on the choice of $d$.
\end{enumerate}
\end{theo}

\begin{proof}
$(1)\Rightarrow(2)$ It suffices to check by induction that for every $m\in\N$ one has
$$\big(z(z\partial_z+1)^s\big)^m\left(\sum_{p\ge1}f_pz^p\right)=
\sum_{p\ge m+1}\Big(\frac{p!}{(p-m)!}\Big)^sf_{p-m}z^p,$$
and then repeat the proof in Theorem~\ref{teoEqM} for the corresponding implication.\par\noindent
$(2)\Rightarrow(3)$ For the modified Borel transform of order $k$ of $\widehat{f}$ and the polynomial $h(z)=1-a_1z-\dots-a_rz^r$ we obtain
$$h(z)\cdot\widetilde{\mathcal{B}}_k\widehat{f}(z)=\sum_{j\ge 1}\Big(\frac{f_{j}}{j!}-\sum_{k=j-r}^{j-1}a_{j-k}\frac{f_k}{k!}\Big)z^j$$
(with $f_0:=0$,\dots,$f_{1-r}:=0$). By (\ref{e228ordens}), the series on the right-hand side defines an entire function of exponential order less than or equal to $k$, and by Remark~\ref{remaordens} we deduce that $\widehat{f}$ is $k$-summable in every direction $d$ which is not an argument of any of the roots of $h$. Moreover, the $k$-sum of $\widehat{f}$ is a solution of the differential equation, again by Watson's Lemma.\par\noindent
$(3)\Rightarrow(1)$ Trivial.
\end{proof}

\begin{coro}\label{coroordenspolynomial}
Let $\widehat{f}(z)=\sum_{p\ge1}f_{p}z^{p}\in\C[[z]]$ be a formal power series and $s=1/k\in\N$. The following statements are equivalent:
\begin{enumerate}
\item[1.-] $\widehat{f}$ is the formal solution of $P\big(z(z\partial_z+1)^s\big)y=Q(z)$, for some polynomials $P,Q\in\C[z]$, with $P(0)\neq 0$.
\item[2.-] There exist $r\in\N$ and $a_1,\dots,a_r\in\C$ such that $\widehat{f}$ verifies the recursion formula
    \begin{equation}\label{equarecursionordens}
\frac{f_{j}}{j!^s}=\sum_{k=j-r}^{j-1}a_{j-k}\frac{f_k}{k!^s}
\end{equation}
for every $j>r$.
\item[3.-] $\widehat{f}$ is $k$-summable in any direction $d$ but the arguments of the roots of a polynomial
$$1-a_rz-\dots-a_1z^r,$$
and its sum, say $f\in\mathcal{O}(S_d)$ for some sector $S_d$ of opening larger than $\pi/k$, is an actual solution of the problem $P\big(z(z\partial_z+1)^s\big)y=g(z)$, for some polynomials $P,Q\in\C[z]$, with $P(0)\neq 0$, which do not depend on $d$.
\end{enumerate}
\end{coro}

\begin{proof}
It resembles that of Corollary~\ref{coroEqMpolynomial}, resting now on Theorem~\ref{teo305ordens}.
\end{proof}

\begin{exam}\label{examEulerseries}
A classical example is given by the Euler series
$$\widehat{g}(z)=\sum_{p\ge0}(-1)^pp!z^{p+1},$$
which turns out to be a formal solution of the differential equation
\begin{equation}\label{eeq}
z^2y'+y=z.
\end{equation}
On the other hand, the function
$$g(z)=\int_{0}^{\infty}\frac{e^{-t/z}}{1+t}dt$$
is holomorphic in the sector $S_{0}(1)$, and it is not difficult to check that $g$ is a solution of (\ref{eeq}). Observe that
$$\int_{0}^{\infty}\frac{e^{-t/z}}{1+t}dt-\sum_{p=0}^{N-1}(-1)^{p}p!z^{p+1}= (-1)^{N}\int_{0}^{\infty}e^{-t/z}\frac{t^{N}}{1+t}dt,\quad z\in S_{0}(1),$$
so that
$$\left| \int_{0}^{\infty}e^{-t/z}\frac{t^{N}}{1+t}dt\right|\le N!|z|^{N+1},\quad z\in S_{0}(1).$$
Indeed, by rotating the half-line of integration the definition of $g$ can be analytically extended to wider sectors in which $g$ is the unique (by Watson's Lemma) holomorphic function admitting $\widehat{g}$ as its Gevrey asymptotic expansion of order 1, and so $\widehat{g}$ is 1-summable in direction $d=0$ with sum $g$.
The procedure in Example~\ref{remaKsummability} may be applied to
$\widehat{f}(z)=\sum_{p\ge0}(-1)^pp!z^{p}$, observe that
$\widehat g(z)=z\widehat{f}(z)$ holds. Then the $1$-sum of $\widehat f$ in direction 0, say $f$, is computed according to~\eqref{equakSumDirectiond}, and we have that $g(z)=zf(z)$ for every $z\in S_0(1)$, as expected.

In order to apply Corollary~\ref{coroordenspolynomial}, one defines $f_{p}=(-1)^{p}p!$ for $p\in\N$ and $a_1=-1$, then
$$f_{j}=ja_1f_{j-1}$$
for every $j>1$. As condition~\eqref{equarecursionordens} holds for $r=1$ and $s=1$, one concludes that $\widehat{f}_0=\sum_{p\ge 1}(-1)^pp!z^{p}=\widehat{f}-1$ is $1$-summable in every direction $d$ but the argument of the zeros of the equation $1-a_1z=0$, i.e., $1+z=0$, which is that of the negative real axis. The polynomials appearing in Corollary~\ref{coroordenspolynomial} are $P(z)=1-a_1z=1+z$ and $Q(z)=f_1z=-z$, and one may easily check that $P(z(z\partial_z+1))\widehat{f}_0 =z^2(\widehat{f}_0)'+z\widehat{f}_0+\widehat{f}_0=-z$, what, taking into account that $\widehat{g}=z(1+\widehat{f}_0)$, amounts to $z^2(\widehat{g})'+\widehat{g}=z$, as it should be the case.
\end{exam}


We mention an easy application of the previous result.

\begin{coro}
Let $a\in\C$ and $h\in\N$. The formal power series
$$\widehat{f}(z)=\sum_{p\ge0} (hp+1)!a^pz^{hp+1}  $$
is $1$-summable in any direction $d$ which is not an argument for any of the $h$-th roots of $1/a$.
\end{coro}
\begin{proof}
Of course, it is enough to reason with the series without its constant term. Let us write $f_p$ for the coefficient of $z^p$ in $\widehat{f}$.
One can easily prove that
$$f_{j}=j(j-1)\dots(j-h+1)af_{j-h}$$
for every $j>h$.
%
So, one can take $r=h$, $a_1=a_2=\dots=a_{h-1}=0$ and $a_h=a$, and apply Corollary~\ref{coroordenspolynomial}. One obtains that $\widehat{f}$ is $1$-summable in every direction but the ones given by the arguments of the complex numbers which verify $1-az^{h}=0$, i.e., the $h$-th roots of $1/a$.
\end{proof}

\begin{rema}\label{remaStokesPhenomGevreyCase}
We will say some words about the Stokes' phenomenon. We start with the situation in Theorem~\ref{teo305ordens} for $s=1$. With the notation in the implication $(2)\Rightarrow(3)$, we have that the formal Borel transform of order 1 of $\widehat{f}$,
$\widehat{\mathcal{B}}_1\widehat{f}(z)=\sum_{p\ge 1}f_pz^p/p!$, equals $b(z)/h(z)$, where
$h(z)=1-a_1z-\dots-a_rz^r$. Consider a singular direction $d_0$ for the 1-summability of $\widehat{f}$, given by an argument of some of the roots of $h$, and suppose the directions in $(d_1,d_2)\setminus\{d_0\}$ are nonsingular. Our aim is to describe the jump between the two solutions of the equation, which we denote by $f_{d_0}^-$ and $f_{d_0}^+$, obtained as 1-sums of $\widehat{f}$ in directions $d^-\in(d_1,d_0)$ and, respectively, in directions $d^+\in(d_0,d_2)$ (the sums in different directions within one of these intervals glue together to define a solution in a wide sector). Observe that, due to our knowledge about the region of convergence of the Laplace transforms involved in the $1$-sum, it is clear that for every $\varepsilon\in(0,1)$ there exists $r_{\varepsilon}>0$ such that the sector $S:=S_{d_0}(1-\varepsilon,r_{\varepsilon})$ is contained in the intersection of the domains of definition of the corresponding functions $f_{d_0}^-$ and $f_{d_0}^+$. Take $z\in S$, then there exist $d^-\in(d_1,d_0)$ and $d^+\in(d_0,d_2)$ such that, according to~\eqref{equakSumDirectiond} for $k=1$,
$$
f_{d_0}^-(z)=\frac{1}{z}\int_{0}^{\infty(d^-)}\frac{b(u)}{h(u)}\,e^{-u/z}\,du,\qquad
f_{d_0}^+(z)=\frac{1}{z}\int_{0}^{\infty(d^+)}\frac{b(u)}{h(u)}\,e^{-u/z}\,du,
$$
and so the jump is
\begin{align}\label{equaJump}
f_{d_0}^+(z)-f_{d_0}^-(z)&=\frac{1}{z}\int_{0}^{\infty(d^+)}\frac{b(u)}{h(u)}\,e^{-u/z}\,du -\frac{1}{z}\int_{0}^{\infty(d^-)}\frac{b(u)}{h(u)}\,e^{-u/z}\,du\nonumber\\
&=\lim_{R\to\infty}\left(\frac{1}{z}\int_{[0,R]e^{id^+}}\frac{b(u)}{h(u)}\,e^{-u/z}\,du -\frac{1}{z}\int_{[0,R]e^{id^-}}\frac{b(u)}{h(u)}\,e^{-u/z}\,du\right),
\end{align}
where $[0,R]e^{id^-}$, respectively $[0,R]e^{id^+}$, stands for the directed segment joining 0 and $Re^{id^-}$, resp. $Re^{id^+}$.
We choose $R$ larger than the modulus of any of the zeros of the function $h$.
Consider now the circular arc $\gamma_R$ parameterized by $\gamma_R(t)=Re^{it}$, $t\in[d^-,d^+]$. Due to the exponential decrease of the function under the integral sign, it is straightforward to check that
$$
\lim_{R\to\infty}\frac{1}{z}\int_{\gamma_R}\frac{b(u)}{h(u)}\,e^{-u/z}\,du=0.
$$
Hence, for the closed path $\Gamma_R:=[0,R]e^{id^-}+\gamma_R-[0,R]e^{id^+}$  we may write, taking into account~\eqref{equaJump},
\begin{equation*}
f_{d_0}^+(z)-f_{d_0}^-(z)=-\lim_{R\to\infty}\frac{1}{z}\int_{\Gamma_R}\frac{b(u)}{h(u)}\,e^{-u/z}\,du.
\end{equation*}
By Cauchy's Theorem, the last integral equals $2\pi i$ times the sum of the residues of the inner function in the singularities lying within the curve $\Gamma_R$. Due to our choice of $R$, this sum remains constant as $R\to\infty$ and it takes into account the residues in the finitely many zeros of the function $h$ lying on the direction $d_0$, so it may be easily computed in concrete cases. In general, we may only be more specific under some simplifying assumption: For example, suppose $\alpha_0$ is the only zero of $h$ with argument $d_0$. In case it is moreover a simple zero, we deduce that
$$
f_{d_0}^+(z)-f_{d_0}^-(z)=-2\pi i\text{res}(\frac{b(u)}{zh(u)}\,e^{-u/z},\alpha_0)=-2\pi i\frac{b(\alpha_0)}{h'(\alpha_0)}\frac{1}{z}e^{-\alpha_0/z}.
$$
If $\alpha_0$ is a zero of order greater than one, we may write
$$
\frac{1}{z}e^{-u/z}=e^{-\alpha_0/z}\sum_{k=0}^\infty\frac{(-1)^k}{k!z^{k+1}}(u-\alpha_0)^k,\quad u\neq\alpha_0;
$$
after expanding $b(u)/h(u)$ in its Laurent expansion around $u=\alpha_0$ and multiplying both expressions, we see that the jump will be a finite linear combination of functions of the type $z^{-m}e^{-\alpha_0/z}$ with $m\in\N$. In either case, since $e^{-\alpha_0/z}$ is exponentially flat in the half-plane bisected by direction $d_0$ (the argument of $\alpha_0$), we observe that the jump also is exponentially flat in $S$, what agrees with the fact that the functions $f_{d_0}^+$ and $f_{d_0}^-$ share the same $1$-Gevrey asymptotic expansion, $\widehat{f}$, in the corresponding sectors.

Similar comments could be made about the jump between neighboring solutions for the case $s\ge 2$, but we will not provide further details.

We may also describe the Stokes' phenomenon in the situation of Theorem~\ref{teoEqM}. Now the formal $\mathfrak{m}$-Borel transform of  $\widehat{f}$ is $b(z)/h(z)$. Following the previous terminology and notation, for every $z$ in the sector $S$ contained in the intersection of the corresponding domains of definition, the solutions at both sides of the singular direction $d_0$ are
$$
f_{d_0}^-(z)=\int_{0}^{\infty(d^-)}\frac{b(u)}{h(u)}\,e(u/z)\,\frac{du}{u},
\qquad
f_{d_0}^+(z)=\int_{0}^{\infty(d^+)}\frac{b(u)}{h(u)}\,e(u/z)\,\frac{du}{u},
$$
and so the jump is
\begin{align*}
f_{d_0}^+(z)-f_{d_0}^-(z)
=\lim_{R\to\infty}\left(\int_{[0,R]e^{id^+}}\frac{b(u)}{h(u)}\,e(u/z)\,\frac{du}{u} -\int_{[0,R]e^{id^-}}\frac{b(u)}{h(u)}\,e(u/z)\,\frac{du}{u}\right).
\end{align*}
Due to the exponential decrease of the function $e(z)$ at infinity in suitable sectors, one can check that
$$
\lim_{R\to\infty}\int_{\gamma_R}\frac{b(u)}{h(u)}\,e(u/z)\,\frac{du}{u}=0,
$$
and so
\begin{equation*}
f_{d_0}^+(z)-f_{d_0}^-(z)=-\lim_{R\to\infty}
\int_{\Gamma_R}\frac{b(u)}{h(u)}\,e(u/z)\,\frac{du}{u}.
\end{equation*}
Again by Cauchy's Theorem, this limit is
$2\pi i$ times the sum of the residues of the inner function in the finitely many zeros of the function $h$ lying on the direction $d_0$.
If we suppose $\alpha_0$ is the only zero of $h$ with argument $d_0$ and it is simple, we see that
$$
f_{d_0}^+(z)-f_{d_0}^-(z)=-2\pi i\text{res}(\frac{b(u)}{uh(u)}\,e(u/z),\alpha_0)=
-2\pi i\frac{b(\alpha_0)}{\alpha_0h'(\alpha_0)}e(\alpha_0/z).
$$
As before, since $e(\alpha_0/z)$ is $\M$-flat in the unbounded sector bisected by $d_0$ and with opening $\pi\omega(\M)$, we deduce that the jump is $\M$-flat, as it should be the case for the difference of two functions sharing a same $\M$-asymptotic expansion.
\end{rema}

\section{Some results for $q$-difference equations}\label{sectqGevrey}

Concerning summability, one can deal with formal power series whose coefficients' growth is not governed by a strongly regular sequence. This is the case of the formal power series
\begin{equation}\label{e296}
\sum_{p\ge1}(-1)^{p}q^{\frac{p(p-1)}{2}}z^{p},
\end{equation}
which is a formal solution of the $q$-difference equation $zy(qz)+y(z)=z$, but whose coefficients have a too fast rate of growth. This behavior is quite natural when regarding $q$-difference equations (see~\cite{malek1},~\cite{ramiszhang} for example). The next definitions can be found in~\cite{zhang}.

\begin{defi}
Let $S=S_{d}(\theta,r)$ be a sector and $q\in\R$ with $q>1$. We say a function $f\in\mathcal{O}(S)$ admits $\widehat{f}(z)=\sum_{n=0}^{\infty}a_nz^n\in\C[[z]]$ as its $q$-Gevrey asymptotic expansion in $S$ if for every proper and bounded subsector $T$ of $S$ there exist $A=A(T)>0$ and $C=C(T)>0$ such that for every $n\in\N_0$ one has
$$\big|f(z)-\sum_{p=0}^{n-1}a_pz^p\big|\le CA^n q^{\frac{n(n-1)}{2}}|z|^n,\quad z\in T.$$
\end{defi}
A formal power series with complex coefficients, $\widehat{f}(z)=\sum_{p\ge0}f_{p}z^{p}$, is said to be $q$-Gevrey (of order 1) if $|f_{p}|\le C A^{p}q^{\frac{p(p-1)}{2}}$ for every $p\in\N_0$ and suitable $C,A>0$. In this case, its formal $q$-Borel transform is defined by
$$ \widehat{\mathcal{B}}_{q}\widehat{f}(z)=\sum_{p\ge0}q^{\frac{-p(p-1)}{2}}f_pz^{p},$$
and it is convergent in a neighborhood of the origin.

We put $\pi_{q}$ for the constant $\ln(q)\prod_{p\ge0}(1-q^{-p-1})^{-1}$, and consider the Jacobi $\theta$ function, holomorphic in $\C^{\star}:=\C\setminus\{0\}$ and defined by
$$\theta(z)=\sum_{p\in\Z}q^{-\frac{p(p-1)}{2}}z^{p},\qquad z\in\C^{\star}.$$
Jacobi theta function naturally appears in the context of $q$-difference equations, for it satisfies the functional equation $zqy(z)=y(qz)$. Moreover, its growth at infinity is stated in the following result (see~\cite{lastramalek} for the details).

\begin{lemma}\label{lemaCrecimJacobi}
There exists $C>0$ such that for every $z\in\C^{\star}$ one has
$$|\theta(z)|\le C(1+|z|)\exp\big(\frac{\log^2(|z|)}{2\log(q)}\big).$$
\end{lemma}

Suppose a formal power series $\widehat{f}(z)=\sum_{p\ge0}f_{p}z^{p}$ is $q$-Gevrey, and its formal $q$-Borel transform converges to a holomorphic function $\varphi$ (defined in a neighborhood of the origin) which can be analytically continued to a function $\Phi$ in an unbounded sector $S$ with bisecting direction $d\in\R$. In addition to this, we assume there exist $C>0,\mu\in\R$ verifying that
\begin{equation}\label{e308}
|\Phi(z)|\le C|z^{\mu}e^{\frac{(\log(z))^{2}}{2\ln(q)}}|,\qquad z\in S.
\end{equation}

In this situation, one can define the function
$$\mathcal{L}_{q}^{d}\Phi(z)=
\frac{1}{\pi_{q}}\int_{0}^{\infty(d)} \frac{\Phi(\xi)}{\theta\left(\frac{\xi}{z}\right)}\frac{d\xi}{\xi},$$
which turns out to be a holomorphic function in some sector $S$ in the Riemann surface of the logarithm $\mathcal{R}$ with opening larger than $2\pi$. Moreover, $\mathcal{L}_{q}^{d}\Phi$ admits $\widehat{f}$ as its $q$-Gevrey asymptotic expansion in $S$. Observe that for the sequence $\M=(q^{\frac{p(p-1)}{2}})_{p\in\N_0}$ the value $\o(\M)$ is $\infty$, what implies that the asymptotic Borel map is not injective, no matter what the opening of the sector is. So, $\mathcal{L}_{q}^{d}\Phi$ is not the only function admitting $\widehat{f}$ as $q$-Gevrey asymptotic expansion in $S$; however, as it is pointed out in~\cite{zhang}, this is the only function with that property in $S$ whose variation is suitably prescribed.
In this situation, one may say that $\widehat{f}$ is $q$-Gevrey summable in direction $d$.

It is important to describe a $q$-analog of the properties of formal Borel transform.

\begin{lemma}
For every $\widehat{f}(z)\in\C[[z]]$, the following formal equality holds:
$$\tau\widehat{\mathcal{B}}_q(\widehat{f})(\tau)= \widehat{\mathcal{B}}_q((z\sigma_q)\widehat{f}(z))(\tau),$$
where $\sigma_q$ stands for the dilation operator $z\mapsto qz$ extended in the natural manner to formal power series.
\end{lemma}

One can adapt the main result in the previous sections for this type of asymptotics.



\begin{theo}\label{teofinal}
Let $ \widehat{f}(z)=\sum_{p\ge1}f_{p}z^p\in\C[[z]]$ be a formal power series.

The following statements are equivalent:
\begin{enumerate}
\item[(1)] $\widehat{f}$ is the formal solution of $P(z\sigma_q)y=g(z)$ for some polynomials $P$ with $P(0)\neq 0$ and some $g\in\C\{z\}$.
\item[(2)] There exist $r\in\N$, $a_1,\dots,a_r\in\C$  and positive constants $C,K$ such that
\begin{equation}\label{e228tris}
\Big|\frac{f_{j}}{q^{j(j-1)/2}}- \sum_{k=j-r}^{j-1}a_{j-k}\frac{f_k}{q^{k(k-1)/2}}\Big|\le \frac{CK^j}{q^{j(j-1)/2}}
\end{equation}
holds for every $j>r$.
\item[(3)] $\widehat{f}$ is $q$-Gevrey summable in any direction $d$ but the arguments of the roots of a polynomial
$$h(z)=1-a_1z-\dots-a_rz^r,$$
and its sum, say $f\in\mathcal{O}(S_d)$ for some sector $S_d$ of opening larger than $2\pi$, is a quasi-solution of the problem $P(z\sigma_q)y=g(z)$, for some polynomial $P\in\C[z]$, with $P(0)\neq 0$, and some $g\in\C\{z\}$ which do not depend on the choice of $d$, in the sense that $P(z\sigma_q)f-g$ admits $\widehat{0}$ as its $q$-Gevrey asymptotic expansion in $S_d$.
\end{enumerate}
\end{theo}
\begin{proof}
We omit many of the details, being analogous to those of the proofs of Theorems~\ref{teoEqM} and~\ref{teo305ordens}. $(1)\Rightarrow(2)$ can be obtained following the same argument. 

$(2)\Rightarrow(3)$ If we put
$h(z)=1-a_1z-\dots-a_rz^r$ and $f_0:=0$,\dots,$f_{1-r}:=0$, it is clear that
$$h(z)\cdot\widehat{\mathcal{B}}_q\widehat{f}(z)=\sum_{j\ge 1}\Big(\frac{f_{j}}{q^{j(j-1)/2}}- \sum_{k=j-r}^{j-1}a_{j-k}\frac{f_k}{q^{k(k-1)/2}}\Big)z^j.$$
By (\ref{e228tris}), the series on the right-hand side, say $b(z)=\sum_{j\ge 1}b_jz^j/q^{j(j-1)/2}$, defines an entire function which, thanks to Lemma~\ref{lemaCrecimJacobi}, is bounded above by $C\Theta(A|z|)$ for suitable $A,C>0$. The function $b(z)/h(z)$ is similarly estimated in any open unbounded sector with vertex at 0 and not containing any of the roots of $h$. This entails $\widehat{f}$ is $q$-Gevrey summable in every direction $d$ which is not an argument of any of the roots of $h$.

Let $f$ be the $q$-Gevrey sum of $\widehat{f}$ in an appropriate direction $d\in\R$, defined in a sector $S_d$ of opening larger than $2\pi$. The coefficient of $z^j$ in $h(z\sigma_q)\widehat{f}$ is
$$f_{j}-q^{j(j-1)/2}\sum_{k=j-r}^{j-1}a_{j-k}\frac{f_k}{q^{k(k-1)/2}}=b_j.$$
Put $g(z):=\sum_{j\ge 1}b_jz^j$, which belongs to $\C\{z\}$.
Since $h(z\sigma_q)\widehat{f}(z)-g(z)$ is the null series, we see that $h(z\sigma_q)f(z)-g(z)$ admits null $q$-Gevrey asymptotic expansion in $\widetilde{S}_d$.\par\noindent 
$(3)\Rightarrow(1)$ Trivial.
\end{proof}

\begin{rema}
In this remark we recall the concept of variation stated in~\cite{zhang}: Given a sector of opening larger than $2\pi$, say $S$, and a function $f$ defined in $S$, we define the variation of $f$ at a point $z\in\{w\in S:we^{2\pi i}\in S\}$ by $\hbox{var}(f)(z)=f(ze^{2\pi i})-f(z)$.

Under the notation adopted in this work, Theorem 8 in~\cite{zhang} reads as follows: Suppose $\widehat{f}$ is such that $\phi:=\widehat{\mathcal{B}}_q\widehat{f}$ is holomorphic in a neighborhood of the origin and it can be analytically continued to a function $\Phi$ defined in a sector $S$ with bisecting direction $d\in\R$, in such a way that (\ref{e308}) holds for some $C>0$ and $\mu\in\R$. Then the function $f(z):=\mathcal{L}_q^d\Phi(z)$ is the only one admitting $\widehat{f}$ as its $q$-Gevrey asymptotic expansion and verifying that
$$\hbox{var } f(x)=\frac{2\pi i}{\pi_q}\sum_{n\in\Z}(-1)^nq^{-n(n+1)/2}\phi(-q^{n}x).$$

So, preserving the notations in the proof of the previous result, one can affirm that in the case that
\begin{equation}\label{e920}
\sum_{n\in\Z}(-1)^{n}q^{-\frac{n(n-1)}{2}}\left[\sum_{j=1}^{r}a_jz^jq^{j-1}\frac{b(-q^{n+j}z)}{h(-q^{n+j})}\right]\equiv0,
\end{equation}
then $f$ is an actual solution of the problem $h(z\sigma_q)f(z)=g(z)$.
\end{rema}

\begin{exam}
Concerning the series (\ref{e296}), which appears in \cite{zhang}, one can put $r=1$ and $a_{1}=(-1)$, so that $f_{j}=a_{1}q^{j-1}f_{j-1}$ for every $j\ge1$, with $f_{j}=(-1)^{j}q^{\frac{j(j-1)}{2}}$. One has that (\ref{e296}) is $q$-Gevrey summable in all directions $d$ but $d=-\pi$.

The equation (\ref{e920}) is satisfied in this example, as expected because the function $f$ constructed in Theorem~\ref{teofinal} turns out to be a solution of the $q$-difference equation $zy(qz)+y(z))=z$. Condition (\ref{e920}) reads as follows in this particular example: $h(z)=1+z$, $b(z)=-z$, and then
$$\sum_{n\in\Z}(-1)^nq^{-\frac{n(n-1)}{2}}\left[(-1)z\frac{b(-q^{n+1}z)}{h(-q^{n+1})}\right]\equiv0$$
if and only if
$$\sum_{n\in\Z}(-1)^nq^{-\frac{n(n+1)}{2}}\frac{q^{n+1}}{1-q^{n+1}}\equiv0.$$
The previous fact is equivalent to $$\sum_{n\in\Z}\frac{(-1)^n}{q^{\frac{n(n+1)}{2}}}\equiv0,$$
which is verified due to the fact that the $n$-th term in the summation is canceled by the term in position $-n-1$, for every $n\in\N_0$.
\end{exam}

\noindent\textbf{Aknowledgements:} A. Lastra and J. Sanz are partially supported by the project PID2019-105621GB-I00 of Ministerio de Ciencia e Innovaci\'on, Spain. J. R. Sendra is supported by the project PID2020-113192GB-I00 (Mathematical Visualization: Foundations, Algorithms and Applications) from the Spanish MICINN.

\vspace{0.4cm}
\noindent Affiliations:

\vspace{0.3cm}

\noindent \textbf{Alberto Lastra and J. Rafael Sendra}\\
Universidad de Alcal\'a. Departamento de F\'isica y Matem\'aticas. Universidad de Alcal\'a,\\ E--28871. Alcal\'a de Henares, Madrid, Spain\\

\noindent \textbf{Javier Sanz}\\
Departamento de \'Algebra, An\'alisis Matem\'atico, Geometr\'ia y Topolog\'ia,\\ Instituto de Investigaci\'on en Matem\'aticas de la Universidad de Valladolid, IMUVA,\\ Facultad de Ciencias, Universidad de Valladolid, 47011 Valladolid, Spain

\end{document}